\documentclass{article}
\usepackage{amsmath,amsfonts,amsthm,amssymb,amscd,color,xcolor,mathrsfs}
\usepackage{hyperref}

\colorlet{darkblue}{blue!50!black}

\hypersetup{
    colorlinks,%
    citecolor=darkblue,%
    filecolor=red,%
    linkcolor=darkblue,%
    urlcolor=magenta,%
    pdfnewwindow=true,%
    pdfstartview={FitH}
}

\newcommand{\e}{\varepsilon}

\newcommand{\R}{{\mathbb R}}
\newcommand{\Z}{{\mathbb Z}}
\newcommand{\IP}{{\mathbb P}}

\newcommand{\E}{{\mathbb E}}

\newcommand{\N}{{\mathbb N}}

\newcommand{\MMM}{\mathfrak{M}}

\newcommand{\XXXX}{{\mathfrak X}}

\newcommand{\bKK}{{\boldsymbol{\KK}}}
\newcommand{\xxi}{{\boldsymbol\xi}}

\newcommand{\eeta}{{\boldsymbol{{\eta}}}}
\newcommand{\mmu}{{\boldsymbol{{\mu}}}}
\newcommand{\nnu}{{\boldsymbol{{\nu}}}}
\newcommand{\zzeta}{{\boldsymbol{{\zeta}}}}

\newcommand{\XXX}{{\boldsymbol{X}}}

\newcommand{\PPP}{{\mathscr P}}

\newcommand{\BB}{{\cal B}}

\newcommand{\DD}{{\cal D}}
\newcommand{\EE}{{\cal E}}
\newcommand{\FF}{{\cal F}}

\newcommand{\KK}{{\cal K}}

\newcommand{\OO}{{\cal O}}
\newcommand{\PP}{{\cal P}}

\newcommand{\sS}{{\cal S}}

\newcommand{\dd}{{\textup d}}

\newcommand{\PPPP}{{\mathfrak P}}

\newcommand{\SSS}{{\mathscr S}}

\newcommand{\Id}{\mathop{\rm Id}\nolimits}

\newcommand{\supp}{\mathop{\rm supp}\nolimits}

\theoremstyle{plain}
\newtheorem*{mt}{Main Theorem}

\newtheorem{theorem}{Theorem}[section]

\newtheorem{lemma}[theorem]{Lemma}
\newtheorem{proposition}[theorem]{Proposition}
\newtheorem{corollary}[theorem]{Corollary}
\theoremstyle{definition}

\theoremstyle{remark}

\numberwithin{equation}{section}

\begin{document}
\title{Markovian reduction and exponential mixing in total variation for random dynamical systems}
\author{Sergei Kuksin\footnote{Universit\'e Paris Cit\'e and Sorbonne Universit\'e, CNRS, IMJ-PRG, 75013 Paris, France \& Peoples Friendship University of Russia (RUDN University) \& Steklov Mathematical Institute of Russian Academy of Sciences, Moscow, Russia \& National Research University Higher School of Economics, Moscow 119048, Russia; Email: \href{mailto:Sergei.Kuksin@imj-prg.fr}{Sergei.Kuksin@imj-prg.fr}} \and Armen Shirikyan\footnote{Department of Mathematics, CY Cergy Paris University, CNRS UMR 8088, 2 avenue Adolphe Chauvin, 95302 Cergy--Pontoise, France; Email: \href{mailto:Armen.Shirikyan@cyu.fr}{Armen.Shirikyan@cyu.fr}}}
\date{}
%\date{\today}
\maketitle

\vspace{-0.5cm}
\begin{abstract}
 The paper deals with the problem of large-time behaviour of trajectories for discrete-time dynamical systems driven by a random noise. Assuming that the phase space is finite-dimensional and compact, and the noise is a Markov process with a transition probability satisfying some regularity hypotheses, we prove that all the trajectories converge to a unique measure in the total variation metric. The proof is based on the Markovian reduction of the system in question and a result on mixing for Markov processes. Then we present an extension of this result to the case of systems driven by stationary noises. The results of this paper were announced in~\cite{KS-rms2024}.
 

\end{abstract}

\tableofcontents

\setcounter{section}{-1}

\section{Introduction}
\label{s0} 
The study of large-time asymptotics of Markovian dynamical systems has a long history and goes back to the  thirties of the last century. Starting from the pioneering works of Kolmogorov~\cite{kolmogorov-1936,kolmogorov-1937} and Doeblin~\cite{doeblin-1938,doeblin-1940}, various methods were developed to prove the existence and uniqueness of a stationary measure and its stability in an appropriate metric; see the books~\cite{hasminskii2012,MT1993,borovkov1998,KS-book}. In the context of discrete-time dynamical systems driven by a random noise~$\{\eta_k\}$, the Markovian hypothesis implies that the random variables~$\eta_k$ should be independent. The case when the dynamics is not Markovian remained essentially open and was studied only for some specific models. Dobrushin~\cite{dobrushin-1968,dobrushin-1974} and Ruelle~\cite{ruelle-1968} proved  uniqueness of a Gibbs state under appropriate hypotheses on transition measures. Their results apply, in particular, to Markov and multiple Markov processes; see~\cite[Section~5.3]{doob1953} for a definition. In the context of a chain of anharmonic oscillators coupled to heat baths, Eckmann, Pillet, and  Rey-Bellet~\cite{EPR-1999} used a Markovian reduction based on the representation of the noise as a stochastic integral and proved convergence to a unique stationary measure (see also Jak\v si\'c--Pillet~\cite{JP-1997,JP-1998} for Hamiltonian systems coupled to one heat bath). A similar idea was used by Hairer~\cite{hairer-2005-fBM} for stochastic differential equations with a dissipative polynomial-type nonlinearity that are driven by a fractional Brownian motion. The recent paper~\cite{KS-2025} deals with the case of dynamical systems driven by stationary noises. It provides sufficient conditions for exponential mixing in the dual-Lipschitz metric for both discrete- and continuous-time cases and applies to a large class of random dynamical systems generated by  dissipative ODEs and PDEs. 

The aim of this paper is to derive simple sufficient conditions for mixing in the {\it total variation metric\/} for discrete-time dynamical systems in a finite-dimensional space that are driven by a Markovian noise or by a stationary noise. Namely, let us denote by~$X$ and~$\KK$ compact subsets of finite-dimensional Euclidean spaces~$H$ and~$E$, respectively, and by $S:H\times E\to H$ a continuous map such that $S(X\times \KK)\subset X$. We consider a random dynamical system of  the form
\begin{equation}\label{RDS}
	u_k=S(u_{k-1},\eta_k), \quad k\ge1,
\end{equation}
supplemented with the initial condition
\begin{equation}\label{IC}
	u_0=u,
\end{equation}
where $u\in X$ is a given point and~$\{\eta_k\}_{k\in\Z}$ is a random process in~$\KK$. We write~$\{u_k(u)\}_{k\in\Z_+}$ for the trajectory defined by~\eqref{RDS}, \eqref{IC}. In general, the random process $\{u_k(u)\}$ is not Markovian. This is the main source of difficulties when studying the large-time asymptotics of trajectories for~\eqref{RDS}. Our approach is based on the following two steps: 
\begin{itemize}\sl
\item Reduction of~\eqref{RDS} to a Markov process in an extended phase space.
\item Investigation of exponential stability of the dynamics generated by the reduced system in the space of probability measures on the extended phase space.
\end{itemize}
Both steps were carried out in detail in the paper~\cite{KS-2025} when the measures on~$X$ are endowed with the {\it dual-Lipschitz metric\/}. The main contribution of this work is the proof of exponential mixing in the stronger norm of {\it  total variation\/}. Let us describe the corresponding result and the main idea of its proof. 

Let us introduce the extended phase space $\XXXX:=X\times \bKK$, where the product space $\bKK:=\KK^{\Z_-}$ is endowed with the Tikhonov topology and distance~\eqref{dist}, and define the process $\eeta_k=(\eta_l,l\le k)$ in~$\bKK$. 
Together with trajectories given by~\eqref{RDS}, we consider the $\XXXX$-valued process $U_k:=(u_k,\eeta_k)$, $k\ge0$. Let us denote by $\{Q(\xxi;\cdot),\xxi\in\bKK\}$ the conditional law of~$\eta_1$, given the past $\eeta_0=\xxi$. Thus, $Q(\xxi;\cdot)$ is a transition probability from the support~$\EE\subset\bKK$ of the law of~$\eeta_0$ to~$\KK$. We shall always assume that one can choose a continuous version of the mapping $\EE\ni\xxi\mapsto Q(\xxi;\cdot)\in\PP(\KK)$, where the space~$\PP(\KK)$ of Borel probability measures on~$\KK$ is endowed with the weak topology. A rather straightforward calculation shows that the family of sequences $\{U_k\}$ corresponding to all possible initial conditions form a Markov process in~$\XXXX$ whose transition probability can be expressed in terms of~$Q(\xxi;\cdot)$ and the map~$S$; see Proposition~1.1 and~\cite[Lemma~1.2]{KS-2025}. Thus, the question of asymptotic behaviour of trajectories for the RDS~\eqref{RDS} reduces to that of the Markov process~$\{U_k\}$.  The price which we pay for this reduction is that the obtained Markov process  is rather degenerate: it is not strongly Markovian, its phase space is infinite dimensional, and the supports of its transition probability are low-dimensional subsets of the phase space. Still, Theorem~1.3 in~\cite{KS-2025} establishes the existence and exponential stability of a unique stationary measure~$\mmu\in\PP(\XXXX)$ for~$\{U_k\}$, provided that the  transition probability~$Q(\xxi;\cdot)$ possesses a Lipschitz-continuous density with respect to the Lebesgue measure on~$H$, and the derivative of the map $S(u,\eta)$ with respect to~$\eta$ is surjective for any $u\in X$ and $\eta\in\KK$. Let us denote by~$\mu$ the projection of~$\mmu$ to~$X$. The main result of this paper can be stated informally as follows (see Theorem~\ref{t-mixing-stationary} for an exact formulation). 

\begin{mt}
Let us assume that the above hypotheses are fulfilled. Then there are positive numbers~$C$ and~$\gamma$ such that, for any $u\in X$, the law of the trajectory of~\eqref{RDS}, \eqref{IC} satisfies the inequality
\begin{equation}
\|\DD(u_k)-\mu\|_{\mathrm{var}}\le Ce^{-\gamma k}, \quad k\ge0.
\end{equation}
\end{mt}

The proof of this result is based on the following observation. Let us denote by $\PPP(U;\cdot)$ the transition probability of the Markov process~$\{U_k\}$. Thus,  $\PPP(U;\cdot)$ is a probability measure on~$\XXXX$ for any $U\in\XXXX$. We write $\Pi:\XXXX\to X$ for the natural projection to the first component and denote by~$\Pi_*$ the corresponding map in the space of probability  measures. The conditions imposed on~$Q$ and~$S$ imply that the measure $\Pi_*\PPP(U;\cdot)$ can be written as $g(U,x)\ell(\dd x)$, where $\ell$ stands for the Lebesgue measure on~$H$, and the function~$g$ is Lipschitz-continuous in both arguments. It follows that if a sequence of measures $\lambda_n\in\PP(\XXXX)$ converges to a limit in the dual-Lipschitz metric, then the projections to~$X$ of their images under the map defined by the transition probability converge in the total variation metric; see Corollary~\ref{c-lip}. The conclusion of the Main theorem is a simple conse\-quence of this property and the exponential stability of~$\mmu$ mentioned above.  

\smallskip
The paper is organised as follows. In Section~\ref{s-DSMS}, we study dynamical systems driven by a {\it Markovian noise\/}, 
formulate for them a theorem on exponential mixing, and discuss its application to ordinary differential  equations with random perturbations. The proof of that theorem is presented in Section~\ref{s-proof-MT}. In Section~\ref{s-dssn}, we state our result in full generality and give its proof. The appendix deals with an auxiliary statement on the images of probability measures under regular maps. 

\subsection*{Acknowledgments} 
The research of AS was supported by the \textit{CY Initiative of Excellence\/} through the grant {\it Investissements d'Avenir\/} ANR-16-IDEX-0008 and by the ANR project DYNACQUS through the grant ANR-24-CE40-5714-01. 

\subsection*{Notation}
We denote by~$\Z$ and~$\R$ the sets of integers and real numbers, respectively,  and write~$\Z_+$ and~$\R_+$ for their subsets consisting of non-negative numbers. For a Polish space\footnote{i.e., a complete separable metric space.}~$X$ with a metric~$\dd_X$, we denote by~$\BB(X)$ its Borel $\sigma$-algebra and by~$\PP(X)$ the set of probability measures on~$(X,\BB(X))$. If $\mu_1,\mu_2\in\PP(X)$, then $\|\mu_1-\mu_2\|_{\mathrm{var}}$ denotes the {\it total variation distance\/} between~$\mu_1$ and~$\mu_2$:
$$
\|\mu_1-\mu_2\|_{\mathrm{var}}=\sup_{\Gamma\in\BB(X)}|\mu_1(\Gamma_1)-\mu_2(\Gamma)|.
$$
We shall also need the {\it dual-Lipschitz distance\/} defined by 
$$
\|\mu_1-\mu_2\|_L^*=\sup_{f}\biggl|\int_Xf\dd\mu_1-\int_Xf\dd\mu_2\biggr|,
$$
where the supremum is taken over all continuous functions $f:X\to\R$ such that 
$$
\sup_{x\in X}|f(x)|+\sup_{0<\dd_X(x,y)\le 1}\frac{|f(x)-f(y)|}{\dd_X(x,y)}\le 1. 
$$
 Given a random variable~$\xi$, we write $\DD(\xi)$ for its law. If $X$ is a vector space and $r>0$ is a number, then $B_X(r)$ denotes the open ball in~$X$ of radius~$r$ centred at zero and ${\overline B}_X(r)$ denotes its closure. If~$A$ is a set, then~${\mathbf1}_A$ stands for its indicator function and~$A^c$ for its complement.

\section{Dynamical systems driven by a Markovian noise}
\label{s-DSMS}

\subsection{A class of random dynamical systems}
Let $X$ and $\KK$ be Polish spaces, $S:X\times\KK\to X$ be a continuous map, and  $\{\eta_k\}_{k\in\Z_+}$ be a Markov process in the space~$\KK$ with a transition probability $\{Q(y;\cdot), y\in \KK\}$. We shall always assume that~$Q$ possesses the Feller property, so that, for any $f\in C_b(\KK)$, the function $y\mapsto \int_\KK Q(y;\dd z)f(z)$ is continuous on~$X$. A simple consequence of the Markov property is that, for any integer $k\ge0$ and any bounded measurable function $f:\KK^{k+2}\to\R$, we have 
\begin{equation}\label{MP-k}
	\E\{f(\eta_0,\dots,\eta_{k+1})\,|\,\sigma(\eta_0,\dots,\eta_k)\}=\int_\KK f(\eta_0,\dots,\eta_{k},z)Q(\eta_k;\dd z).
\end{equation}
Indeed, relation~\eqref{MP-k} is obvious for functions of the form $f(y_0,\dots,y_{k+1})=g(y_0,\dots,y_k)h(y_{k+1})$ and can easily be derived from that particular case with the help of standard arguments.

Let consider the RDS~\eqref{RDS}, \eqref{IC} in~$X$. Note that the law of a trajectory $\{u_k(u),k\ge0\}$ will not change if in~\eqref{RDS} we replace~$\{\eta_k\}$ with another process~$\{\tilde\eta_k\}$ with the same distribution, which may be defined on a different probability space. Together with~\eqref{RDS}, we study an RDS in the extended phase  space $\XXXX=X\times\KK$, whose points will be denoted by $U=(v,\xi)$. To describe it, we consider a Markov process $\{\zeta_k, k\ge0\}$ in~$\KK$, defined on a filtered space $(\Omega,\FF,\{\FF_k\}_{k\in\Z_+})$ with a probability measure~$\IP$ such that the time-$1$ transition probability is equal to~$Q(y;\cdot)$. Let us note that, in contrast to~$\{\eta_k\}$, the Markov process~$\{\zeta_k\}$ can be issued from any initial state $y\in\KK$. We now define the RDS
\begin{equation}\label{RDS-extended}
	U_k=\SSS(U_{k-1},\zeta_k), \quad k\ge1,
\end{equation}
where $\SSS:\XXXX\times \KK\to\XXXX$ is given by the formula $\SSS(U,\zeta)=\bigl(S(v,\zeta),\zeta\bigr)$ for $U=(v,\xi)\in\XXXX$ and $\zeta\in \KK$. Thus, writing $U_l=(v_l,\xi_l)$ for $l\in\Z_+$, we obtain $U_k=(S(v_{k-1},\zeta_k),\zeta_k)$. To define a trajectory of~\eqref{RDS-extended}, we add the initial condition
\begin{equation}\label{IC-extended}
	U_0=V:=(v,\xi),
\end{equation}
where $(v,\xi)\in\XXXX$ is an $\FF_0$-measurable random variable. In what follows, we denote the trajectory of~\eqref{RDS-extended}, \eqref{IC-extended} by $\{U_k(V)\}_{k\in\Z_+}$. 

\begin{proposition}\label{p-markov}
	The family of trajectories $\{U_k(V)\}_{k\in\Z_+}$ corresponding to all possible initial conditions $V=(v,\xi)\in\XXXX$ form a homogeneous Markov process in~$\XXXX$ with respect to the filtration~$\{\FF_k\}$ associated with the process~$\{\zeta_k\}$. Moreover, the following two assertions hold. 
	\begin{itemize}
		\item [\hypertarget{a}{\bf(a)}] The time-$1$ transition probability $\{\PPP(U;\cdot),U\in\XXXX\}$ of the process~$\{U_k\}$ is given by the relation
\begin{equation}\label{TF-extended}
\PPP(U;\cdot)=\SSS_*\bigl((v,\xi),Q(\xi;\cdot)\bigr),
\end{equation}
where the map $\SSS((v,\xi),\cdot):\KK\to\XXXX$ sends~$\zeta$ to $(S(v,\zeta),\zeta)$.
		\item[\hypertarget{b}{\bf(b)}] If an  $\FF_0$-measurable random initial state has the form~$V=(u,\xi)$, where $u\in X$ is deterministic and $\DD(\xi)=\DD(\eta_0)$, then for any $k\ge0$ the trajectory $\{U_k(V)\}$, $U_k=(v_k,\xi_k)$, satisfies the relation
\begin{equation}\label{laws}
\DD([v_0,\dots,v_k])=\DD([u_0,\dots,u_k]),
\end{equation}
where $\{u_k\}_{k\in\Z_+}$ is the trajectory of~\eqref{RDS}, \eqref{IC}. 
\end{itemize}
\end{proposition}

\begin{proof}
	To prove that the trajectories $\{U_k(V)\}$ form a Markov process, it suffices to show that, for any $k\in\Z_+$ and any bounded measurable function $f:\XXXX\to\R$, we have 
	\begin{equation}\label{MP-proof}
		\E\,\bigl\{f(U_{k+1}(V)\,|\,\FF_k\bigr\}=\hat f\bigl(U_{k}(V)\bigr),
	\end{equation}
where $\hat f:\XXXX\to\R$ is a measurable function. To see this, let us note that if~\eqref{MP-proof} is established for bounded continuous functions, then the	general case will follow by a standard argument. To establish~\eqref{MP-proof} for $f\in C(\XXXX)$, we first assume that the function $(v,\xi)\mapsto f(S(v,\xi),\xi)$ can be represented in the form $g(v)h(\xi)$. In this case, using~\eqref{RDS-extended} and the Markov property for~$\{\zeta_k\}$, we can write 
\begin{multline}
\E\,\bigl\{f(U_{k+1}(V)\,|\,\FF_k\bigr\}
	=\E\,\bigl\{g(v_k)h(\eta_{k+1})\,|\,\FF_k\bigr\}
	=g(v_k)\,\E\,\bigl\{h(\eta_{k+1})\,|\,\FF_k\bigr\}\\
	=g(v_k)\int_\KK Q(\zeta_k;\dd y)h(y)
	=\int_\KK Q(\zeta_k;\dd y)f(S(v_k,y),y).\label{MP-particular}
\end{multline}
This proves~\eqref{MP-proof} in the above-mentioned particular case. Relation~\eqref{MP-proof} in the general situation can be proved by a simple approximation argument, based on the density of linear combinations of the functions of the form $g(v)h(\xi)$ in the space of continuous functions.  

Taking $k=1$ and $f={\mathbf1}_\Gamma$ in~\eqref{MP-particular}, we arrive at relation~\eqref{TF-extended} for the transition probability. Finally, assertion~\hyperlink{b}{(b)} can be established by induction in~$k$, applying the argument used in the proof of~\cite[Lemma~1.2]{KS-2025}.
\end{proof}

Proposition~\ref{p-markov} allows one to reduce the question of the large-time asymptotics of trajectories for the RDS~\eqref{RDS} driven by a Markovian noise to a similar problem for the Markov process~\eqref{RDS-extended}. Roughly speaking, Theorem~\ref{t-mixing-markov} below proves the exponential stability of the dynamics defined by~\eqref{RDS}, provided that the Markov process~$\eta_k$ possesses appropriate mixing properties. 

\subsection{Main result and scheme of its proof}
\label{s-MR-scheme}
In this section, we denote by~$H$ and~$E$  finite-dimensional Euclidean spaces with norms~$|\cdot|$ and~$\|\cdot\|$, respectively, and by $X\subset H$ and $\KK\subset E$ compact subsets containing the zeros of the respective spaces. Let us assume that we are given  a $C^2$-smooth map $S:H\times E\to H$ such that $S(0,0)=0$, $S(X\times\KK)\subset X$, and the following two properties hold:

\begin{description}
	\item [\hypertarget{(D)}{Dissipativity.}] \sl 
	For any $\e>0$ there is an integer $n_\e\ge1$ such that 
\begin{gather}
	|S_{n_\e}(u;0,\dots,0)|\le \e\quad\mbox{for $u\in X$},\label{stability}
\end{gather}
where $\{S_n(u;\zeta_1,\dots,\zeta_n)\}$ denotes  the trajectory of system~\eqref{RDS}, \eqref{IC} in which $\eta_j=\zeta_j$. 

	\item [\hypertarget{(C)}{Controllability.}] 
The map $(D_\eta S)(0,0):E\to H$ is surjective and the map $(D_u S)(0,0):H\to H$ is an isomorphism.
\end{description}

Let us consider the RDS~\eqref{RDS}, in which~$\{\eta_k\}$ is a Markov process in~$E$ (not necessarily stationary) with a Feller continuous time-$1$ transition probability~$Q(y;\Gamma)$ such that $\supp\DD(\eta_0)\subset\KK$ and $\supp Q(y;\cdot)\subset\KK$ for any $y\in \KK$. We impose the following two hypotheses:
\begin{description}\sl 
	\item [\hypertarget{(M)}{Minorisation.}]
There exists a number $r>0$ and a bounded continuous function $\rho:{\overline B}_E(r)\times E\to\R_+$ such that $\rho(0,0)>0$ and 
\begin{equation}\label{TFn-LB}
	Q(y;\dd z)\ge\rho(y,z)\ell(\dd z)\quad\mbox{for $y\in \KK\cap {\overline B}_E(r)$},
\end{equation}
where $\ell\in\PP(E)$ is the Lebesgue measure on~$E$. 
\item [\hypertarget{(SR)}{Strong recurrence.}] For any $\delta>0$ there is an integer $l\ge1$ and number $\varkappa>0$ such that
\begin{equation}\label{Q-to-zero}
		Q_l\bigl(y;B_E(0,\delta)\bigr)\ge \varkappa\quad\mbox{for any $y\in\KK$},
\end{equation}
where $\{Q_k(y;\cdot), k\in\Z_+,y\in E\}$ is the transition probability for~$\{\eta_k\}$. 
\end{description}
These two conditions ensure that the Markov process~$\{\eta_k\}$ converges exponentially in the total variation metric to a stationary measure $\nu\in\PP(E)$; e.g., see~\cite[Chapters~15 and~16]{MT1993} or~\cite[Theorem~3.1]{shirikyan-2017}. The following theorem, which is the main result of this section, shows that the same is true for the RDS~\eqref{RDS}. Let us introduce the direct product $\XXX=X^\Z$, endowed with the Tikhonov topology, and denote by~$\PP_s(\XXX)$ the set of measures in~$\PP(\XXX)$ that are invariant under translations. 

\begin{theorem}\label{t-mixing-markov}
	Let us assume that the \hyperlink{(D)}{Dissipativity}, \hyperlink{(C)}{Controllability}, \hyperlink{(M)}{Minorisation}, and \hyperlink{(SR)}{Strong recurrence} hypotheses hold. Then there is a measure $\mmu\in\PP_s(\XXX)$ and a number~$\gamma>0$ such that, for any initial state $u\in X$, the corresponding trajectory~$\{u_k\}_{k\in\Z_+}$ of~\eqref{RDS}, \eqref{IC} satisfies the inequality 
	\begin{equation}\label{expo-mixing}
\|\DD([u_k,\dots,u_{k+m}])-\mmu_m\|_{\mathrm{var}}\le C_me^{-\gamma k}, \quad k\ge0,
	\end{equation}
where $\mmu_m$ denotes the projection of~$\mmu$ to~$X^{m+1}$, and~$C_m>0$ is a number not depending on the initial condition. 
\end{theorem}

A proof of Theorem~\ref{t-mixing-markov} is given in the next section. Here we describe briefly the main idea. In view of Proposition~\ref{p-markov}, it suffices to establish the exponential mixing for the Markov process associated with the RDS~\eqref{RDS-extended}, which 
is considered on the compact invariant set~$\XXXX=X\times\KK$. To this end, it suffices to check the following {\it strong recurrence\/} and {\it coupling\/} properties for some point $\widehat U\in\XXXX$; see~\cite[Section~15.4]{MT1993} or~\cite[Section~3.1]{shirikyan-2017}: 
\begin{description}\sl 
	\item [\hypertarget{(Rec)}{(Rec)}]
For any $r>0$, there is an integer~$m\ge1$ and a number $p>0$ such that 
\begin{equation}\label{recurrence}
	\PPP_m\bigl(U,B_\XXXX(\widehat U,r)\bigr)\ge p\quad\mbox{for any $U\in\XXXX$},
\end{equation}
where $\PPP_m(U;\Gamma)$ is the time-$m$ transition probability for~\eqref{RDS-extended}; cf.~\eqref{TF-extended}.  
	\item [\hypertarget{(Cou)}{(Cou)}] There exists a number $\e>0$ and an integer $n\ge1$ such that  
\begin{equation}\label{coupling}
	\|\PPP_n(U_1;\cdot)-\PPP_n(U_2;\cdot)\|_{\mathrm{var}}\le 1-\e \quad\mbox{for}\quad    U_1,U_2\in {\overline B}_\XXXX(\widehat U,\e). 
\end{equation}
\end{description}
The validity of~\hyperlink{(Rec)}{(Rec)} with $\widehat U=(0,0)$ is a consequence  of~\eqref{stability} and the fact that the trajectories of the Markov process~$\{\zeta_k\}$ with the transition probability~$Q_k(y,\cdot)$ may stay close to zero on arbitrarily long time intervals.  Checking Hypothesis~\hyperlink{(Cou)}{(Cou)} is more delicate. It is based on the fact that, for $U\in \XXXX$ sufficiently close to zero,  the time-$2$ transition probability $\PPP_2(U;\cdot)$ is minorised by a positive measure. This will be proved with the help of a result on the image of measures under smooth maps. 

\subsection{Application}
\label{s-ode-kicked}
Let $H$ be the Euclidean space~$\R^d$ with the inner product $\langle\cdot,\cdot\rangle$ and the corresponding norm~$|\cdot|$. Consider the following ordinary differential equation (ODE) in~$H$:
\begin{equation}\label{ODE}
	\dot x=V(x)+\eta(t).
\end{equation}
Here $V:\R^d\to\R^d$ is a $C^2$-smooth vector field such that the following two conditions hold:
\begin{description}\sl 
\item[\hypertarget{D}{Dissipation.}] There are positive numbers~$C$ and~$c$ such that
\begin{equation}\label{dissipation-V}
	\langle V(x),x\rangle\le -c\,|x|^2+C \quad\mbox{for $x\in H$},
\end{equation}
\item[\hypertarget{(St)}{Stable equilibrium.}] All the trajectories of the vector field~$V(x)$ converge to the point $x=0$ as $t\to+\infty$, uniformly with respect to the initial conditions in any ball. 
\end{description}
Let us note that the latter condition is certainly fulfilled if~\eqref{dissipation-V} holds, and the vector field satisfies the inequality
\begin{equation}
	\langle V(x),x\rangle<0 \quad\mbox{for $x\in H\setminus\{0\}$}.
\end{equation}
As for $\eta$, we assume that it is a random process of the form 
\begin{equation}\label{eta-t}
	\eta(t)=\sum_{k=0}^\infty \eta_k\delta(t-k),
\end{equation}
where $\delta(t)$ is the Dirac mass at zero and~$\{\eta_k\}$ is a homogeneous Markov process in~$H$. 

Let us denote by~$\varphi:H\to H$ the time-$1$ shift along trajectories of Eq.~\eqref{ODE} with $\eta\equiv0$. It follows from~\eqref{dissipation-V} that
\begin{equation}\label{dissipation}
	|\varphi(x)|\le \beta\,|x|+C_1\quad
	\mbox{for $x\in H$},
\end{equation}
where $\beta\in(0,1)$ and $C_1>0$ are some numbers not depending on~$x$. The trajectories of~\eqref{ODE} have jumps at integer times, and we normalise them to be right-continuous. Denoting by~$x_k$ the value of a solution~$x(t)$ at time $t=k$, we see that relation~\eqref{RDS} holds with a map~$S$ given by $S(x,\eta)=\varphi(x)+\eta$. The following theorem describes the large-time asymptotics of trajectories of~\eqref{ODE} restricted to integer times.

\begin{theorem}\label{t-ODE}
In addition to the above-mentioned  properties, let us assume that the Markov process~$\{\eta_k\}$ satisfies the \hyperlink{(M)}{Minorisation} and \hyperlink{(SR)}{Strong recurrence} hypotheses. Then there is measure $\mmu\in\PP_s(H^\Z)$, supported by~$\KK^\Z$, and a number $\gamma>0$ such that, for any integer $m\ge0$ and some constants $C_m>0$, we have 
		\begin{equation}\label{ode-mixing}
\|\DD([x_k,\dots,x_{k+m}])-\mmu_m\|_{\mathrm{var}}\le C_me^{-\gamma k}\bigl(1+|x_0|\bigr), \quad k\ge0,
	\end{equation}
where $x_0\in H$ is an arbitrary initial condition, and~$\mmu_m$ denotes the projection of~$\mmu$ to~$H^{m+1}$. 
\end{theorem}

\begin{proof}
	We claim that if~$R>0$ is sufficiently large, then the ball $X={\overline B}_H(R)$ is invariant under the dynamics and absorbs the trajectories issued from any ball~$B_H(M)$ in a finite time of order~$\ln M$. To see this, we note that, in view of~\eqref{dissipation}, 
	\begin{equation}\label{lyapunov}
		|S(x,\eta)|\le \beta\,|x|+|\zeta|+C_1\quad\mbox{for any $x,\zeta\in H$}. 
	\end{equation} 
Denoting by~$\varkappa>0$ the maximum of the norms of the elements of~$\KK$, we deduce from~\eqref{lyapunov} the invariance of the ball~${\overline B}_H(R)$, provided that $R\ge \frac{\varkappa+C_1}{1-\beta}$. Setting $R\ge \frac{2(\varkappa+C_1)}{1-\beta}$, we also obtain the absorption property. 

What has been said implies that~\eqref{ode-mixing} will follow immediately from inequality~\eqref{expo-mixing}, in which $\{u_k\}=\{x_k\}$ is the trajectory corresponding to an arbitrary initial point in~$X={\overline B}_H(R)$. Indeed, in view of~\eqref{lyapunov}, any trajectory will hit the ball ${\overline B}_H(R)$ in a time of order $\ln|x_0|$, whence the required inequality can be obtained by a simple calculation. Thus, we only need to check the hypotheses of Theorem~\ref{t-mixing-markov} in the current setting. 

Let us set $E=H=\R^d$. The \hyperlink{(D)}{Dissipativity} follows from the \hyperlink{(St)}{Stable equilibrium} hypothesis. The maps entering \hyperlink{(C)}{Controllability} have the following form:
$$
(D_\eta S)(0,0)=\Id_{H},\quad (D_u S)(0,0)=D\varphi(0). 
$$
The surjectivity of $(D_\eta S)(0,0)$ is obvious, and the fact that $(D_u S)(0,0)$ is an isomorphism follows from well-known results on solutions of ODEs. Finally, the \hyperlink{(M)}{Minorisation} and \hyperlink{(SR)}{Strong recurrence} hypotheses are postulated, and the proof of the theorem is complete. 
\end{proof}

\section{Proof of Theorem~\ref{t-mixing-markov}}
\label{s-proof-MT}
\subsection*{Reduction}
By Proposition~\ref{p-markov}, the law of the vector~$[u_k,\dots,u_{k+m}]$ coincides with that of $[v_k,\dots,v_{k+m}]$, where $v_k$ is the first component of the trajectory~$U_k$ of~\eqref{RDS-extended} corresponding to an initial condition with the law $\delta_u\otimes\DD(\eta_0)$. Therefore, it suffices to prove inequality~\eqref{expo-mixing} for~$\{v_k\}$. To this end, we shall show that the Markov process defined by the RDS~\eqref{RDS-extended} possesses a unique stationary measure $\nu\in\PP(\XXXX)$, and there are positive numbers~$\gamma$ and~$C$ such that, for any $V\in\XXXX$ and $k\ge0$, 
\begin{equation}\label{mixing-for-U}
	\|\DD(U_k)-\nu\|_{\mathrm{var}}\le Ce^{-\gamma k}. 
\end{equation}
Assuming that this is proved, the proof can be concluded as follows. Let us denote by $\{V_k\}_{k\in\Z_+}$ a trajectory of~\eqref{RDS-extended} whose initial law is equal to~$\nu$ and write~$\nnu\in\PP_s(\XXXX^{\Z_+})$   for the law of the random variable $(V_k,k\in\Z_+)$. By the Markov property, it follows from~\eqref{mixing-for-U} that 
\begin{equation}\label{mixing-for-pathU}
	\|\DD([U_k,\dots,U_{k+m}])-\nnu_m\|_{\mathrm{var}}\le C_me^{-\gamma k}, \quad k\ge0, 
\end{equation}
where $\nnu_m\in\PP(\XXXX^{m+1})$ stands for the marginal of~$\nnu$ in~$\XXXX^{m+1}$.  Denoting by~$\mmu$ the projection of~$\nnu$ to the $X$-component, we obtain~\eqref{expo-mixing} as a straightforward consequence of~\eqref{mixing-for-pathU}. Thus, we need to establish~\eqref{mixing-for-U}. This will be done if we prove the validity of properties~\hyperlink{(Rec)}{(Rec)} and~\hyperlink{(Cou)}{(Cou)}. 

\subsection*{Checking~\hyperlink{(Rec)}{(Rec)}}
We wish to prove that~\eqref{recurrence} holds with $\widehat U=(0,0)$. To this end, we first show that inequality~\eqref{recurrence} is valid with some $p=p_1>0$ for initial points of the form $U=(v,0)$. Indeed, let us fix any $r>0$ and use~\eqref{stability} to find $m\ge1$ such that 
\begin{equation}\label{to-zero}
	|S_m(v;0,\dots,0)|\le \frac{r}{2}\quad\mbox{for $v\in X$}. 
\end{equation}
Since $S:H\times E\to H$ is continuous, so is $S_m:H\times E^m\to H$. Using~\eqref{to-zero} and the compactness of~$X$, we can find $\delta>0$ such that, if vectors $\xi_1,\dots,\xi_m\in E$ satisfy the inequality $\|\xi_k\|<\delta$ for $1\le k\le m$, then 
$$
	|S_m(v;\xi_1,\dots,\xi_m)|<r\quad\mbox{for any $v\in X$}. 
$$
Hence, the required inequality will be established if we prove that, for any $U=(v,0)$ with $v\in X$, 
\begin{equation}\label{ineq}
P_m(v,\delta):=\IP_U\bigl\{\|\zeta_k\|<\delta\mbox{ for }1\le k\le m\bigr\}\ge p_1>0,
\end{equation}
where the subscript~$U\in\XXXX$ indicates that we consider the trajectory issued from~$U$. This is a consequence of the observation that $\zeta_1,\dots,\zeta_m$ is the trajectory issued from~$\zeta_0=0$, and therefore inequality~\eqref{ineq} with a suitable $p_1=p_1(m)>0$ is true for any integer $m\ge1$ in view of the \hyperlink{(M)}{Minorisation} hypothesis. 

We now prove the validity of~\eqref{recurrence} for any $U\in \XXXX$. In view of the Feller property, compactness of~$X$, and inequality~\eqref{recurrence} with $U=(v,0)$, there is $\delta>0$ such that, for any $U=(v,\zeta)\in\XXXX$ with $\zeta\in B_E(0,\delta)$, we have 
$$
\PPP_m\bigl(U;B_\XXXX(0,r)\bigr)\ge p_1/2.
$$ 
If we find an integer $l\ge1$ and a number $p_2>0$ such that 
\begin{equation}\label{transition-X-0}
\PPP_l\bigl(U;X\times B_E(0,\delta)\bigr)\ge p_2\quad\mbox{for any $U\in \XXXX$}, 	
\end{equation}
then the required inequality~\eqref{recurrence} with $m$ replaced by $l+m$ will follow from the Kolmogorov--Chapman relation. The validity of~\eqref{transition-X-0} follows immediately from~\eqref{Q-to-zero} and relation~\eqref{TF-extended} for the transition probability. 

\subsection*{Checking~\hyperlink{(Cou)}{(Cou)}}
We shall prove the validity of~\eqref{coupling} for $n=2$. To this end, it suffices to find a non-negative measure~$\lambda\ne0$ on~$\XXXX$ and a number~$\delta>0$ such that
\begin{equation}\label{lower-bound}
\PPP_2(U;\cdot)\ge\lambda\quad\mbox{for $U\in {\overline B}_\XXXX(0,\delta)$};	
\end{equation}
in this case, \eqref{coupling} will be true with $\e=\lambda(\XXXX)$. This type of estimates for the images of probability measures under smooth maps is well known; see~\cite[Theorem~9.6.10]{bogachev2010}, \cite[Theorem~2.2]{AKSS-aihp2007} and~\cite[Theorem~2.4]{shirikyan-jfa2007}. However, inequality~\eqref{lower-bound} is not a consequence of known results, and we present a complete proof. 

To estimate the measure $\PPP_2(U;\cdot)$ from below, we consider integrals against it of bounded measurable functions $f:H\times E\to\R_+$. Namely, it follows  from~\eqref{RDS-extended} that, for any $U=(v,\zeta)\in\XXXX$, we have
\begin{align}
(\PPPP_2f)(U)&=
\int_\XXXX f(V)\PPP_2(U;\dd V)\notag\\
&=\int_E\int_E f\bigl(S_2(v;z_1,z_2),z_2\bigr)
Q(\zeta;\dd z_1)Q(z_1;\dd z_2),
\label{P2f}
\end{align}
where $\{\PPPP_k\}$ is the semigroup in $C(\XXXX)$ associated with the transition probability~$\PPP_k(U;\cdot)$. Roughly speaking, we wish to make a change of variable $w=S_2(v,z_1,z_2)$. To this end, we need the following lemma. 

\begin{lemma}\label{l-diffeo}
	Under the hypotheses of the theorem, there is a subspace $F=\{z^+\}$ in~$E$ of dimension~$\dim H$, with the orthogonal complement~$F^\bot=\{z^-\}$, and a number $\alpha>0$ such that, for any $(v,z_1^-,z_2)\in {\overline B}_H(\alpha)\times {\overline B}_{F^\bot}(\alpha)\times {\overline B}_E(\alpha)$ the $C^1$-smooth map 
	$$
	B_F(\alpha)\ni z_1^+\mapsto S_2(v,z_1^++z_1^-,z_2)\in H
	$$
	is a diffeomorphism onto its image that can be extended to a global diffeomorphism $\varPhi(\cdot; v, z_1^-,z_2)$ from~$F$ onto~$H$. 
	\end{lemma}

\begin{proof}
	We first note that any local diffeomorphism can be extended to a global one in the following sense: if~$F$ is a vector space of dimension~$\dim H$ and $\varPhi:O_F\to O_H$ is a diffeomorphism between some open sets~$O_F$ and~$O_H$, then for any $z_0^+\in O_F$ there is a diffemorphism $\tilde\varPhi:F\to H$ that coincides with~$\varPhi$ in a neighbourhood of~$z_0^+$. 	To see this, it suffices to take a smooth truncating function $\chi:F\to\R$ equal to~$1$ for $|z|\le\frac12$ and vanishing for $|z|\ge1$, and to set 
	$$
	\tilde\varPhi(z^+)=\varPhi(z_0^+)+\Bigl(1-\chi\bigl(\tfrac{z^+-z_0^+}{\e}\bigr)\Bigr)A(z^+-z_0^+)+\chi\bigl(\tfrac{z^+-z_0^+}{\e}\bigr)\,\bigl(\varPhi(z^+)-\varPhi(z_0^+)\bigr), 
	$$ 
	where $A$ is the derivative of~$\varPhi$ at $z_0^+$, and $\e>0$ is a small parameter. Obviously, $\tilde\varPhi$ is $C^2$-smooth. It equals~$\varPhi$ near $z_0^+$, and 
	$$
	\tilde\varPhi(z^+)=\varPhi(z_0^+)+A(z^+-z_0^+)\quad
	\mbox{for $z^+\notin B^\e:=\{z^+\in F:|z^+-z_0^+|\le\e\}$}.
	$$
	Hence, if $z_1^+\ne z_2^+$, and both points lie outside~$B^\e$, then  $\tilde\varPhi(z_1^+)\ne \tilde\varPhi(z_2^+)$. Furthermore, since~$\varPhi$ is $C^2$-smooth, a direct calculation shows that $\tilde\varPhi(z_1^+)\ne \tilde\varPhi(z_2^+)$ if at least one of the points $z_1^+\ne z_2^+$ belongs to~$B^\e$ and $\e\ll 1$. Thus, $\tilde\varPhi$ is a $C^2$ diffeomorphism of~$F$. Obviously, the same construction works if~$\varPhi$ depends on a parameter~$q$ (where $q=(v,z_1^-,z_2)$ in our case) if we assume that~$q$ varies in a small neighbourhood of a given point. 
	
	Thus, we need to prove that there is a finite-dimensional subspace $F\subset E$ such that the function $z_1^+\mapsto S_2(v,z_1^++z_1^-,z_2)$ is a diffeomorphism of a small open ball~$B_F(\delta)$ onto its image, provided that the point $(v,z_1^-,z_2)\in X\times F^\bot\times E$ belongs to  a small neighbourhood of zero in the product space. To this end, let us recall that $S_2(v;z_1,z_2)=S(S(v,z_1),z_2)$, so that 
	$$
	(D_{z_1}S_2)(v;z_1,z_2)=(D_uS)(S(v,z_1),z_2)\circ (D_{z_1}S)(v,z_1). 
	$$
The \hyperlink{(C)}{Controllability} hypothesis and the relation $S(0,0)=0$ imply that operator $(D_{z_1}S_2)(0;0,0):E\to H$ is surjective. Therefore, we can find a subspace~$F\subset E$ of dimension~$\dim H$ such that the restriction of~$A$ to~$F$ is an isomorphism. The required assertion follows now from the implicit function theorem. 
\end{proof}

We now prove~\eqref{lower-bound}.  The \hyperlink{(M)}{Minorisation} hypothesis implies that, if $\delta>0$ is  sufficiently small, then for $\zeta,z_1\in B_E(\delta)$ we have 
\begin{equation}\label{zero-approx}
	Q(\zeta;\dd z_1)Q(z_1;\dd z_2)
	\ge \chi(z_1,z_2)\ell(\dd z_1)\ell(\dd z_2),
\end{equation}
where $\chi:E\times E\to\R_+$ is a continuous function, positive at zero and vanishing outside the direct product $B_E(\delta)\times B_E(\delta)$. We can assume that $\delta\le\alpha$, where $\alpha>0$ is the number defined in Lemma~\ref{l-diffeo}. Denoting by~$F\subset E$ and~$\varPhi$  the subspace and map constructed there and using~\eqref{P2f} and~\eqref{zero-approx}, for any bounded  measurable function $f:H\times E\to\R_+$ we can write 
\begin{align*}
(\PPPP_2f)(U)
&\ge\int_E\int_E
f\bigl(S_2(v;z_1,z_2),z_2\bigr)
\chi(z_1,z_2)\ell(\dd z_1)\ell(\dd z_2)\notag\\
&= \int\limits_F\ell(\dd z_1^+) \int\limits_{F^\bot\times E}f\bigl(\varPhi(z_1^+;v,z_1^-,z_2),z_2\bigr)\chi(z_1,z_2)\ell(\dd z_1^-)\ell(\dd z_2), 
\end{align*}
where we used the fact that $\varPhi$ and~$S_2$ coincide on the support of~$\chi$. Setting $q=(v,z_1^-,z_2)$ and carrying out the change of variable $w=\varPhi(z_1^+;q)$, we see that the right-most term in the above inequality can be written as 
$$
\int_F\ell(\dd w)\int_E f(w,z_2)\biggl\{\int_{F^\bot}D(w;q)\chi(\varPsi(w;q)+z_1^-,z_2)\ell(\dd z_1^-)\biggr\}\ell(\dd z_2),
$$
where $\varPsi(\cdot;q)$ stands for inverse of~$\varPhi(\cdot;q)$ and  $D(w;q)$  is the Jacobian of the transformation $z_1^+=\varPsi(w,q)$. We have thus proved inequality~\eqref{lower-bound} with a measure $\lambda=\lambda_v$ depending on~$v$: 
$$
\lambda_v(\dd w,\dd z_2)
=\biggl\{\int_{F^\bot}D(w;q)\chi(\varPsi(w;q)+z_1^-,z_2)\,\ell(\dd z_1^-)\biggr\}\ell(\dd w)\ell(\dd z_2). 
$$
It remains to note that the density of~$\lambda_v$ (i.e., the expression in the curly brackets) is a continuous function of its arguments $(v,w,z_2)$ that is positive at zero. We conclude that~\eqref{lower-bound} holds with a measure~$\lambda$ whose density is given by $\gamma\,{\mathbf1}_{W(\gamma)}(w,z_2)$, where $\gamma>0$ is a small number and $W(\gamma)=B_H(\gamma)\times B_E(\gamma)$. This completes the proof  of Theorem~\ref{t-mixing-markov}. 

\section{Dynamical systems driven by a stationary noise}
\label{s-dssn}
Now let us consider the RDS~\eqref{RDS}, where~$\{\eta_k\}$ is a stationary process in the Euclidean space~$E$ such that $\DD(\eta_k)$ is supported by a compact set~$\KK$. We introduce the direct product~$\bKK:=\KK^{\Z_-}$ and endow it with the metric 
\begin{equation}\label{dist}
\dd(\xxi,\xxi')=\sum_{k\in\Z_-}\iota^k|\xi_k-\xi_k'|, \quad \xxi=(\xi_k,k\in\Z_-), \quad \xxi'=(\xi_k',k\in\Z_-),
\end{equation}
where $\iota>1$ is an arbitrary fixed number. For any $\xxi\in\bKK$, we denote by~$Q(\xxi;\cdot)$ the conditional law of~$\eta_1$ given the past $\eta_k=\xi_k$, $k\in\Z_-$. Setting
\begin{equation}\label{denote}
\sigma=\DD(\{\eta_k,\in\Z_-\}), \quad \EE:=\supp\sigma\subset \bKK, 
\end{equation}
we recall that $Q(\xxi;\cdot)$ is uniquely defined on~$\EE$ up to a set of zero $\sigma$-measure. We shall always assume that the following property holds.

\begin{description}\sl 
	\item [\hypertarget{FP}{Feller property}.] The function $\xxi\mapsto Q(\xxi;\cdot)$ can be chosen to be continuous from~$\EE$ to the space~$\PP(\KK)$. 
\end{description}
Defining $\eeta_k:=(\eta_l,l\le k)$ for $k\in\Z_-$, we see that $\{\eeta_k\}$ is a homogeneous Markov process in~$\EE$ with the time-$1$ transition probability $\delta_\xxi\otimes Q(\xxi;\cdot)\in\PP(\EE)$. Defining the map $\sS:H\times \bKK\to H$ by the formula 
$$
\sS(u,\eeta)=S(u,\eta_0)\quad\mbox{for $u\in H$, $\eeta=(\eta_k,k\in\Z_-)\in\bKK$},
$$ 
it is straightforward to see that $\{u_k\}$ is a trajectory of~\eqref{RDS} if and only if 
\begin{equation}\label{RDS-reduced}
	u_k=\sS(u_{k-1},\eeta_k), \quad k\ge1. 
\end{equation}
We have thus reduced the RDS~\eqref{RDS} driven by a stationary noise to one driven by the stationary  Markovian noise~$\eeta_k$. The main difference with the problem studied in Section~\ref{s-MR-scheme} is that the noise space is now infinite dimensional. Combining the above-mentioned reduction with some methods developed to study Markov processes in infinite-dimensional spaces, we have proved in~\cite{KS-2025} the uniqueness of a stationary measure and its exponential stability in the dual-Lipschitz metric. We now recall the corresponding hypotheses and establish a stronger result about the convergence to the limiting measure in the total variation metric. 

Let us introduce the following three hypotheses, which are a simplified version of those in~\cite[Section~2.1]{KS-2025}. For any integer $k\ge1$, we define a transition  probability ${\mathbf Q}_k(\xxi;\cdot)$ from~$\EE$ to~$\KK^k$ by the relation 
\begin{equation}\label{conditional-law}
	{\mathbf Q}_k(\xxi;\dd y_1,\dots,\dd y_k)
	=Q(\xxi;\dd y_1) Q(\xxi_1;\dd y_2)\cdots Q(\xxi_{k-1};\dd y_k),
\end{equation}
where $\xxi_l =(\xxi, y_1, \dots, y_l)$. Note that $	{\mathbf Q}_k(\xxi;\cdot)$ is the conditional law of $(\eta_1,\dots,\eta_k)$ given a past $\eta_l=\xi_l$, $l\in\Z_-$. For any $n,s\in\N$, we denote by ${\mathbf Q}_s^n(\xxi;\cdot)\in\PP(\KK^n)$ the projection of ${\mathbf Q}_{s+n}(\xxi;\cdot)$ to the last~$n$ components.

\begin{description}\sl 
	\item \hypertarget{LC}{{\bf Linearised controllability.}}
The map $D_\eta S(u,\eta):E\to H$ is surjective for any $u\in X$ and $\eta\in\KK$.
	\item \hypertarget{LR}{{\bf Lipschitz regularity.}}
There is a function $\rho:\EE\times E\to\R_+$ Lipschitz continuous in both variables such that, for any $\xxi\in\EE$, the measure $Q(\xxi;\cdot)$ is absolutely continuous with respect to the Lebesgue measure~$\ell$ with the density~$\rho(\xxi,\cdot)$: 
\begin{equation}\label{Q-density}
Q(\xxi;\dd y)=\rho(\xxi,y)\ell(\dd y). 
\end{equation}
\item \hypertarget{RZ}{{\bf Recurrence to zero.}} For any $n\ge1$ and $\delta>0$, there is $s\ge1$ such that
	\begin{equation}\label{rec-to-zero}
	\inf_{\xxi\in\EE}{\mathbf Q}_s^n\bigl(\xxi;\OO_\delta(\mathbf{0}_n)\bigr)>0,
	\end{equation}
	where $\OO_\delta({\mathbf0}_n)$ stands for the $n$-fold product of the ball $B_E(0,\delta)$.
\end{description}
The following result is an amplification of Theorem~2.5 in~\cite{KS-2025} for the case when the phase and control spaces are finite-dimensional.\footnote{We recall that the result of~\cite{KS-2025}  states the exponential convergence of $\DD([u_k,\dots,u_{k+m}])$ to a measure~$\mmu_m$ in the dual-Lipschitz metric.} 

\begin{theorem}\label{t-mixing-stationary}
	Suppose that the RDS~\eqref{RDS} satisfies the hypotheses of \hyperlink{(D)}{Dissipativity}, \hyperlink{LC}{Linearised controllability}, \hyperlink{LR}{Lipschitz regularity}, and \hyperlink{RZ}{Recurrence to zero}. Then there is a measure $\mmu\in\PP_s(\XXX)$ and a number~$\gamma>0$ such that, for any initial state $u\in X$, inequality~\eqref{expo-mixing} holds for~$\{u_k\}$ defined by~\eqref{RDS}, \eqref{IC} with a constant~$C_m>0$ not depending on~$u$. 
\end{theorem}

\begin{proof}
We shall derive convergence~\eqref{expo-mixing} from a similar result for the dual-Lipschitz metric  in~\cite{KS-2025}. We first describe a construction used there to reduce the system under study  to a Markov process in a larger space.  
 
\smallskip
{\it Step~1: Reduction\/}. Recalling~\eqref{denote},  we define the space $\XXXX=X\times\EE$ with points $U=(v,\xxi)$ and introduce a map~$\SSS:\XXXX\times\EE\to\XXXX$ by the formula (cf.\ \eqref{RDS-extended})
\begin{equation}\label{SSS-map}
\SSS(U,\eta)=(S(v,\eta),(\xxi,\eta)\bigr). 	
\end{equation}
Let $\{\zeta^\xxi,\xxi\in\EE\}$ be a measurable random field in~$E$ such that $\DD(\zeta^\xxi)=Q(\xxi;\cdot)$ for any $\xxi\in\EE$. We denote by $(\Omega,\FF,\IP)$ the direct product of countably many copies of the probability space on which~$\zeta^\xxi$ is defined, write $\omega=(\omega_0,\omega_1,\dots)$ for its points, and consider the following RDS in~$\XXXX$ (cf.\ the construction in~\cite[Section~1.2]{KS-2025}):
\begin{equation}\label{RDS-stationary}
	U_k=\SSS\bigl(U_{k-1},\zeta^{\xxi_{k-1}}(\omega_k)\bigr), \quad k\ge1, 
\end{equation}
where $U_k=(v_k,\xxi_k)$. We supplement the system with the initial condition~\eqref{IC-extended}, in which~$V$ is a random variable in~$\XXXX$ depending only on~$\omega_0$. Relation~\eqref{RDS-stationary} implies that~$U_k$ is an $\XXXX$-valued random variable depending only on~$\omega^k=(\omega_0,\dots,\omega_k)$. A key observation allowing one to reduce the study of~\eqref{RDS} to that of~\eqref{RDS-stationary} is the following result established in~\cite[Lemma~1.4]{KS-2025}.

\begin{lemma}
	Let $V=V(\omega_0)$ be an $\XXXX$-valued random variable whose law is equal to~$\delta_u\otimes\sigma$ {\rm(}see~\eqref{denote}{\rm)}, where $u\in X$ is an arbitrary point. Then, for any integer $n\ge0$, we have
	$$
	\DD\bigl([v_0,\dots,v_n]\bigr)=\DD\bigl([u_0,\dots,u_n]\bigr),
	$$
	where $\{u_k,k\in\Z_+\}$ denotes the trajectory of~\eqref{RDS}, \eqref{IC}, and~$v_k$ stands for the first component of the process~$U_k$ defined by system~\eqref{RDS-stationary} with the initial condition~$U_0=V$. 
\end{lemma}

Thus, it suffices to establish convergence~\eqref{expo-mixing} for trajectories of~\eqref{RDS-stationary}. To make the main idea more transparent, we first consider the case $m=0$. 

\smallskip
{\it Step~2: Particular case\/}. In view of~\cite[Theorem~2.5]{KS-2025}, the Markov process associated with~\eqref{RDS-stationary} is exponentially mixing in the dual-Lipschitz metric. It means that there is a unique stationary measure $\MMM\in\PP(\XXXX)$, and for any initial condition $U\in\XXXX$, the corresponding trajectory~$\{U_k\}$ of~\eqref{RDS-stationary} satisfies the inequality
\begin{equation}\label{DU-M}
	\|\DD(U_k)-\DD(W_k)\|_L^*\le Ce^{-\gamma k}, \quad k\ge0,
\end{equation}
where~$\{W_k\}$ stands for the trajectory issued from an initial condition distributed as~$\MMM$, so $\DD(W_k)=\MMM$ for $k\ge0$. Let us denote by~$\Pi:\XXXX\to X$ the projection to the $X$-component of~$\XXXX$. We claim that
\begin{equation}\label{PiDU-M}
	\|\Pi_*\DD(U_k)-\Pi_*\DD(W_k)\|_{\mathrm{var}}\le C'e^{-\gamma k}, \quad k\ge0;
\end{equation}
this coincides with~\eqref{expo-mixing} when $m=0$. 

To prove~\eqref{PiDU-M}, we note that
$$
\Pi_*\PPP_1(U;\cdot)=\DD\bigl(S(v,\zeta^\xxi)\bigr) 
\quad\mbox{for any $U=(v,\xxi)\in\XXXX$}. 
$$ 
Combining the \hyperlink{LC}{Linearised controllability} and \hyperlink{LR}{Lipschitz regularity}  with Theorem~\ref{t-image-measure} in Appendix, we see that $\Pi_*\PPP_1(U;\cdot)=g(U,x)\ell(\dd x)$ for $U\in\XXXX$, where $g(U,x)$ is a Lipschitz function of its argument. Using the Markov property, we conclude that 
$$
\Pi_*\DD(U_{k})=\bigl(\E\,g(U_{k-1},x)\bigr)\ell(\dd x), \quad 
\Pi_*\DD(W_{k})=\bigl(\E\,g(W_{k-1},x)\bigr)\ell(\dd x). 
$$
The required inequality~\eqref{PiDU-M} follows now from~\eqref{DU-M} (with~$k$ replaced by~$k-1$) and Corollary~\ref{c-lip}. 

\smallskip
{\it Step~3: General case\/}. Following the same idea as in the case~$m=0$, we now derive a formula for $\Pi_*\DD(U_{1},\dots,U_{m})$ in terms of the law of~$Q(\xxi;\cdot)$. Namely, for any integer $m\ge0$ and any initial point $U=(v,\xxi)\in\XXXX$, we consider the segment~$[v_1,\dots,v_m]$ of the $\Pi$-projection of the trajectory $\{U_k=(v_k,\xxi_k),k\ge0\}$ for~\eqref{RDS-stationary} issued from~$U$. In view of~\eqref{SSS-map} and~\eqref{RDS-stationary}, we have
$$
[v_1,\dots,v_m]=\bigl[S\bigl(v,\zeta_1^\xxi\bigr),S_2\bigl(v;\zeta_1^\xxi,\zeta_2^{\xxi_1}\bigr),\dots,S_m\bigl(v;\zeta_1^\xxi,\dots,\zeta_m^{\xxi_{m-1}}\bigr)\bigr],
$$
where $\zeta_k^{\xxi_{k-1}}=\zeta^{\xxi_{k-1}}(\omega_k)$. Now note that the law of the vector $(\zeta_1^\xxi,\dots,\zeta_m^{\xxi_{m-1}})$ is equal to ${\mathbf Q}_m(\xxi;\cdot)$; see~\eqref{conditional-law}. It follows that $\DD([v_1,\dots,v_m])$ is equal to the image of ${\mathbf Q}_m(\xxi;\cdot)$ under the map 
$$
\vec{S}_m(U,\cdot):E^m\to H^m, \quad [\eta_1,\dots,\eta_m]\mapsto 
\bigl[S(v,\eta_1),\dots,S_m(v;\eta_1,\dots,\eta_m)\bigr]
$$
(which depends on~$U$ as a parameter). In view of inequality~\eqref{DU-M} and Corollary~\ref{c-lip} (where~$H$ and~$E$ are replaced with~$H^m$ and~$E^m$, respectively), the required convergence will be established if we prove that the \hyperlink{(S)}{Surjectivity} and \hyperlink{(LD)}{Lipschitz density} hypotheses are satisfied for the map~$\vec{S}_m$ and measure~${\mathbf Q}_m$. The second property follows immediately from the observation that 
$$
{\mathbf Q}_m(\xxi;\dd {\mathbf y}_m)=
\rho(\xxi,y_1)\,\rho\bigl((\xxi,y_1),y_2\bigr)\cdots\rho\bigl((\xxi,y_1,\dots,y_{m-1}),y_m\bigr)\,\ell(\dd{\mathbf y}_m), 
$$
where ${\mathbf y}_m=(y_1,\dots,y_m)$, and~$\ell$ is the Lebesgue measure on~$E^m$; see~\eqref{conditional-law} and~\eqref{Q-density}. 

To prove the surjectivity of $D\vec S_m(U,\cdot)$, we abbreviate $\eeta_k=[\eta_1,\dots,\eta_k]$ and note that 
$$
\vec{S}_{m+1}(\eeta_{m+1})=\bigl[v_m,S(v_m,\eta_{m+1})\bigr],
$$
where $v_m=\vec{S}_{m}(\eeta_{m})$. It follows that 
\begin{equation}\label{relation}
D\vec{S}(\eeta_{m+1})\zzeta_{m+1}=\bigl[A_m\zzeta_m, D_vS(v_m,\eta_{m+1})A_m\zzeta_m+D_\eta S(v_m,\eta_{m+1})\zeta_{m+1}\bigr],
\end{equation}
where $A_m=D\vec{S}_m(\eeta_m):E^m\to H^m$. For $m=1$, $D\vec{S}_1(\eeta_1)$ is surjective by the \hyperlink{LC}{Linearised controllability}. Arguing by induction, assume that for some $m\ge1$ the map $D\vec{S}_m(\eeta_m):E^m\to H^m$ is surjective. Then~\eqref{relation} and the \hyperlink{LC}{Linearised controllability} imply the surjectivity of $D\vec{S}_{m+1}(\eeta_{m+1}):E^{m+1}\to H^{m+1}$.  This completes the proof of the theorem. 
\end{proof}

\section{Appendix: image of probability measures under regular maps}
Let~$\XXXX$ be a Polish space, let~$E=\{y\}$ and~$H=\{x\}$ be finite-dimensional spaces, and let $\KK\subset E$ be a compact subset. We shall write~$\ell$ for the Lebesgue measure on any finite-dimensional space. Let us consider a continuous map $F:\XXXX\times E\to H$ such that, for any $U\in\XXXX$, the function $F(U,\cdot):E\to H$ is twice continuously differentiable in the neighbourhood of~$\KK$, its derivatives up to the second order are bounded uniformly in~$U$, and the differential~$D_yF(U,y)$ is a Lipschitz function of~$(U,y)$. Let $\{\lambda(U,\cdot)\}_{U\in\XXXX}$ be a transition probability from~$\XXXX$ to~$E$ such that the support of $\lambda(U,\cdot)$ is included into~$\KK$ for any $U\in\XXXX$. We impose the following hypotheses:

\begin{description}\sl 
	\item [\hypertarget{(S)}{Surjectivity.}]
For any $(U, y)\in\XXXX\times\KK$, the linear operator $(D_ y F)(U, y):E\to H$ is surjective.
	\item [\hypertarget{(LD)}{Lipschitz density.}] There is a Lipschitz continuous function $\rho:\XXXX\times E\to\R_+$ such that,  for any $U\in\XXXX$, we have
\begin{equation}\label{density-lambda}
	\lambda(U,\dd y)=\rho(U,y)\ell(\dd y).
\end{equation}
\end{description}

For any $U\in\XXXX$ and $\lambda\in\PP(E)$, we shall denote by $F_*(U,\lambda)\in\PP(H)$ the image of the measure~$\lambda$ under the map $ y\mapsto F(U, y)$ acting from~$E$ to~$H$. 

\begin{theorem}\label{t-image-measure}
	Let us assume that the \hyperlink{(S)}{Surjectivity} and \hyperlink{(LD)}{Lipschitz density} hypotheses are satisfied. Then there is a function $g:\XXXX\times H\to\R_+$ Lipschitz-continuous in both arguments such that, for $U\in\XXXX$, we have 
\begin{equation}\label{density-image}
	F_*\bigl(U,\lambda(U,\cdot)\bigr)
	=g(U,x)\,\ell(\dd x). 
\end{equation}
\end{theorem}

\begin{proof}
	Justification of~\eqref{density-image} is based on well-known ideas (cf.\ \cite[Chapter~9]{bogachev2010} or~\cite{AKSS-aihp2007}), and therefore we shall skip some of the details. In what follows, when talking about a family of measures $\lambda(U,\cdot)$, we shall always assume that it satisfies the \hyperlink{(LD)}{Lipschitz density} hypothesis. 

Obviously, it suffices to prove that, for any $\widehat U\in\XXXX$ and $\hat y\in\KK$, there is a number $r>0$ and Lipschitz continuous function
$$
g:{\overline B}_\XXXX(\widehat U,r)\times H\to\R_+
$$
such that~\eqref{density-image} holds for $U\in {\overline B}_\XXXX(\widehat U,r)$ and any measure $\lambda(U,\cdot)$ with support included in~$\KK\cap {\overline B}_E(\hat y,r)$. Let us fix any $\widehat U\in\XXXX$, $\hat y\in\KK$ and set $\hat x=F(\widehat U,\hat y)$. By assumption, the linear operator $(D_y F)(\widehat U,\hat y):E\to H$ is surjective. Therefore, we can write~$E$ as the direct sum $E_1\dotplus E_2$, with the corresponding decompositions $y=y_1+y_2$ and $\hat y=\hat y_1+\hat y_2$ , so that the derivative $(D_{y_1}F)(\widehat U,\hat y):E_1\to H$ is an isomorphism. Using a parameter version of the implicit function theorem, we can find positive numbers~$r$ and~$\delta$ such that, for any $\widehat U\in {\overline B}_\XXXX(\widehat U,r)$ and $y_2\in {\overline B}_{E_2}(\hat y_2,2r)$, the map $y_1\mapsto F(U,y_1+y_2)$ is a diffeomorphism of the ball ${\overline B}_{E_1}(\hat y_1,2r)$ onto its image included in ${\overline B}_H(\hat x,\delta)$. We denote by~$y_1=G(U,x,y_2)$ its inverse, which is a Lipschitz continuous function of~$(U,x,y_2)$ and is~$C^1$ smooth with uniformly bounded derivatives as a function of~$(x,y_2)$.

We now derive an explicit expression for the measure $F_*(U,\lambda(U,\cdot))$, assuming that $\supp\lambda(U,\cdot)\subset {\overline B}_E(\hat y,r)$. To this end, we introduce a smooth truncating function $\chi(y)$ equal to~$1$ on the product $W_r:=B_{E_1}(\hat y_1,r)\times B_{E_2}(\hat y_2,r)$ and vanishing outside~$W_{2r}$. In view of~\eqref{density-lambda}, for any $f\in C_b(H)$ and $U\in {\overline B}_\XXXX(\widehat U,r)$, we can write 
\begin{multline*}
	\bigl\langle f, F_*\bigl(U,\lambda(U,\cdot)\bigr)\bigr\rangle =\int_E f\bigl(F(U,y)\bigr)\lambda(U,\dd y)\\
	=\int_{E_1}\int_{E_2} \chi(y_1+y_2)f\bigl(F(U,y_1+y_2)\bigr)\rho(U,y_1+y_2)\ell(\dd y_1)\ell(\dd y_2).
\end{multline*}
Abbreviating $Z=(U,x,y_2)$, carrying out the change of variable $y_1=G(Z)$ (so that $x=F(U,y_1+y_2)$), and denoting by $J(Z)$ its Jacobian with respect to~$x$, we write the right-most term as
$$
\int_H f(x)\biggl\{\int_{E_2} \chi(G(Z)+y_2)\rho(U,G(Z)+y_2)J(Z)\ell(\dd y_2)\biggr\}\ell(\dd x).
$$
We thus obtain relation~\eqref{density-image}, in which~$g(U,x)$ is given by the expression in the curly brackets. The regularity of the functions entering that expression implies that~$d$ is Lipschitz-continuous with respect to its arguments. This completes the proof of the theorem. 
\end{proof}

Let us define a map $\varPsi: \PP(\XXXX)\to \PP(H)$  by the relation
\begin{equation*}
\varPsi(\nu)=\int_\XXXX F_*\bigl(U,\lambda(U,\cdot)\bigr)\nu(\dd U), 
\end{equation*}
where $\nu\in\PP(\XXXX)$ is arbitrary. The following simple result provides sufficient conditions for the Lipschitz continuity of~$\varPsi$. 

\begin{corollary}\label{c-lip}
	Under the hypotheses of Theorem~\ref{t-image-measure}, there is a number $C>0$  such that 
	\begin{equation}\label{lipschitz-measures}
	\bigl\|\varPsi(\nu_1)-\varPsi(\nu_2)\bigr\|_{\mathrm{var}}\le C\,\bigl\|\nu_{1}-\nu_{2}\bigr\|_L^*
	\quad\mbox{for any $\nu_1,\nu_2\in\PP(\XXXX)$}.  
	\end{equation}
\end{corollary}

\begin{proof}
	In view of~\eqref{density-image}, for any $f\in C_b(H)$ and $\nu\in \PP(\XXXX)$ we have 
$$
\langle f,\varPsi(\nu)\rangle=\int_\XXXX \biggl\{\int_H f(x)g(U,x)\ell(\dd x)\biggr\}\nu(\dd U).
$$
Let $\MMM\in\PP(\XXXX\times\XXXX)$ be an arbitrary measure whose marginals are equal to~$\nu_1$ and~$\nu_2$. Recalling a well-known formula for the total variation distance, we obtain
\begin{multline*}
\bigl\|\varPsi(\nu_1)-\varPsi(\nu_2)\bigr\|_{\mathrm{var}}=\frac12
\sup_f\,\bigl|\langle f,\varPsi(\nu_1)\rangle-\langle f,\varPsi(\nu_2)\rangle\bigr|\\
=\frac12\sup_f\,\biggl|\,\int_{\XXXX\times\XXXX} \biggl\{\int_H f(x)\bigl(g(U_1,x)-g(U_2,x)\bigr)\ell(\dd x)\biggr\}\MMM(\dd U_1,\dd U_2)\biggr|,
\end{multline*}
where the supremum is taken over all continuous functions $f:H\to\R$ whose absolute value is bounded by~$1$. Using the Lipschitz property of~$g(U,x)$ with respect to~$U$, we derive
	$$
\bigl\|\varPsi(\nu_1)-\varPsi(\nu_2)\bigr\|_{\mathrm{var}}
\le C\int_{\XXXX\times\XXXX}\dd_\XXXX(U_1,U_2)\,\MMM(\dd U_1,\dd U_2). 
	$$
Taking the infimum over all $\MMM\in\PP(\XXXX\times\XXXX)$ with marginals~$(\nu_1,\nu_2)$ and using the Kantorovich--Rubinstein theorem (see~\cite[Theorem~11.8.2]{dudley2002}), together with the fact that the Kantorovich and dual-Lipschitz distances are equivalent on~$\PP(\XXXX)$, we arrive at the required inequality~\eqref{lipschitz-measures}. 
\end{proof}

\addcontentsline{toc}{section}{Bibliography}
\def\cprime{$'$} \def\cprime{$'$}
  \def\polhk#1{\setbox0=\hbox{#1}{\ooalign{\hidewidth
  \lower1.5ex\hbox{`}\hidewidth\crcr\unhbox0}}}
  \def\polhk#1{\setbox0=\hbox{#1}{\ooalign{\hidewidth
  \lower1.5ex\hbox{`}\hidewidth\crcr\unhbox0}}}
  \def\polhk#1{\setbox0=\hbox{#1}{\ooalign{\hidewidth
  \lower1.5ex\hbox{`}\hidewidth\crcr\unhbox0}}} \def\cprime{$'$}
  \def\polhk#1{\setbox0=\hbox{#1}{\ooalign{\hidewidth
  \lower1.5ex\hbox{`}\hidewidth\crcr\unhbox0}}} \def\cprime{$'$}
  \def\cprime{$'$} \def\cprime{$'$} \def\cprime{$'$}
\providecommand{\bysame}{\leavevmode\hbox to3em{\hrulefill}\thinspace}
\providecommand{\MR}{\relax\ifhmode\unskip\space\fi MR }
\providecommand{\MRhref}[2]{%
  \href{http://www.ams.org/mathscinet-getitem?mr=#1}{#2}
}
\providecommand{\href}[2]{#2}


\begin{thebibliography}{AKSS07}

\bibitem[AKSS07]{AKSS-aihp2007}
A.~Agrachev, S.~Kuksin, A.~Sarychev, and A.~Shirikyan, \emph{On
  finite-dimensional projections of distributions for solutions of randomly
  forced 2{D} {N}avier--{S}tokes equations}, Ann. Inst. H. Poincar\'e Probab.
  Statist. \textbf{43} (2007), no.~4, 399--415.

\bibitem[Bog10]{bogachev2010}
V.~I. Bogachev, \emph{Differentiable {M}easures and the {M}alliavin
  {C}alculus}, Mathematical Surveys and Monographs, vol. 164, American
  Mathematical Society, Providence, RI, 2010.

\bibitem[Bor98]{borovkov1998}
A.~A. Borovkov, \emph{{Ergodicity and Stability of Stochastic Processes}}, John
  Wiley \& Sons, Ltd., Chichester, 1998.

\bibitem[Dob68]{dobrushin-1968}
R.~L. Dobru{\v{s}}in, \emph{Description of a random field by means of
  conditional probabilities and conditions for its regularity}, Teor.
  Verojatnost. i Primenen \textbf{13} (1968), 201--229.

\bibitem[Dob74]{dobrushin-1974}
\bysame, \emph{Conditions for the absence of phase transitions in
  one-dimensional classical systems}, Math. USSR-Sb. \textbf{22} (1974), no.~1,
  28--48.

\bibitem[Doe38]{doeblin-1938}
W.~Doeblin, \emph{Expos\'e de la th\'eorie des cha\^\i nes simples constantes
  de {M}arkov \`a un nombre fini d'\'etats}, Rev. Math. Union Interbalkan
  \textbf{2} (1938), 77--105.

\bibitem[Doe40]{doeblin-1940}
\bysame, \emph{{\'E}l\'ements d'une th\'eorie g\'en\'erale des cha\^\i nes
  simples constantes de {M}arkoff}, Ann. Sci. \'Ecole Norm. Sup. (3)
  \textbf{57} (1940), 61--111.

\bibitem[Doo53]{doob1953}
J.~L. Doob, \emph{{Stochastic Processes}}, John Wiley \& Sons, New York, 1953.

\bibitem[Dud02]{dudley2002}
R.~M. Dudley, \emph{Real {A}nalysis and {P}robability}, Cambridge University
  Press, Cambridge, 2002.

\bibitem[EPR99]{EPR-1999}
J.-P. Eckmann, C.-A. Pillet, and L.~{Rey-Bellet}, \emph{Non-equilibrium
  statistical mechanics of anharmonic chains coupled to two heat baths at
  different temperatures}, Comm. Math. Phys. \textbf{201} (1999), no.~3,
  657--697.

\bibitem[Hai05]{hairer-2005-fBM}
M.~Hairer, \emph{Ergodicity of stochastic differential equations driven by
  fractional {B}rownian motion}, Ann. Probab. \textbf{33} (2005), no.~2,
  703--758.

\bibitem[JP97]{JP-1997}
V.~Jak\v{s}i\'{c} and C.-A. Pillet, \emph{Ergodic properties of the
  non-{M}arkovian {L}angevin equation}, Lett. Math. Phys. \textbf{41} (1997),
  no.~1, 49--57.

\bibitem[JP98]{JP-1998}
\bysame, \emph{Ergodic properties of classical dissipative systems. {I}}, Acta
  Math. \textbf{181} (1998), no.~2, 245--282.

\bibitem[Kha12]{hasminskii2012}
R.~Khasminskii, \emph{{Stochastic Stability of Differential Equations}},
  Springer, Heidelberg, 2012.

\bibitem[Kol36]{kolmogorov-1936}
A.~Kolmogoroff, \emph{Zur {T}heorie der {M}arkoffschen {K}etten}, Math. Ann.
  \textbf{112} (1936), no.~1, 155--160.

\bibitem[Kol37]{kolmogorov-1937}
A.~N. Kolmogorov, \emph{Markov chains with a countable state space}, Moscow
  Univ. Math. Bulletin \textbf{1} (1937), no.~3, 1--16 (in Russian).

\bibitem[KS12]{KS-book}
S.~Kuksin and A.~Shirikyan, \emph{Mathematics of {T}wo-{D}imensional
  {T}urbulence}, Cambridge University Press, Cambridge, 2012.

\bibitem[KS24]{KS-rms2024}
S.~B. Kuksin and A.~R. Shirikyan, \emph{Mixing in stochastic dynamical systems
  with stationary noise}, Russian Math. Surveys \textbf{79} (2024), no.~6,
  1098--1100.

\bibitem[KS25]{KS-2025}
S.~Kuksin and A.~Shirikyan, \emph{Mixing for dynamical systems driven by
  stationary noises}, Preprint (2025),
  \href{https://arxiv.org/abs/2502.11699}{arXiv:2502.11699}.

\bibitem[MT93]{MT1993}
S.~P. Meyn and R.~L. Tweedie, \emph{Markov {C}hains and {S}tochastic
  {S}tability}, Springer-Verlag, London, 1993.

\bibitem[Rue68]{ruelle-1968}
D.~Ruelle, \emph{Statistical mechanics of a one-dimensional lattice gas}, Comm.
  Math. Phys. \textbf{9} (1968), 267--278.

\bibitem[Shi07]{shirikyan-jfa2007}
A.~Shirikyan, \emph{Qualitative properties of stationary measures for
  three-dimensional {N}avier--{S}tokes equations}, J. Funct. Anal. \textbf{249}
  (2007), 284--306.

\bibitem[Shi17]{shirikyan-2017}
A.~R. Shirikyan, \emph{{Controllability implies mixing. I. Convergence in the
  total variation metric}}, Russian Math. Surveys \textbf{72} (2017), no.~5,
  939--953.

\bibitem[Str93]{stroock1993}
D.~Stroock, \emph{Probability. {A}n {A}nalytic {V}iewpoint}, Cambridge
  University Press, Cambridge, 1993.

\end{thebibliography}
\end{document}